\documentclass[a4paper,11pt]{article}

\usepackage{fullpage} 
\usepackage{hyperref}
\usepackage{algorithm}
\usepackage{algpseudocode}
\usepackage{amsmath}
\usepackage{amsfonts}
\usepackage{amssymb}
\usepackage{amsthm}
\usepackage{mathrsfs}
\usepackage{textcomp}
\usepackage{stmaryrd}
\usepackage{latexsym} 
\usepackage{enumitem}
\usepackage{caption}
\usepackage{morefloats}
\usepackage[utf8]{inputenc}
\usepackage[T2A]{fontenc}
\usepackage[all]{xy}

\theoremstyle{plain}
\newtheorem{theorem}{Theorem}[section]

\newtheorem{lemma}[theorem]{Lemma}

\newtheorem{corollary}[theorem]{Corollary}

\theoremstyle{definition}
\newtheorem{definition}[theorem]{Definition}

\theoremstyle{plain}

\theoremstyle{definition}

\newcommand{\reduces}{\ensuremath{\to^*}}

\makeatletter

\newcommand{\Rmnum}[1]{\expandafter\@slowromancap\romannumeral #1@}
\makeatother

\newcommand{\CL}{\mathrm{CL}}
\newcommand{\FV}{\mathrm{FV}}
\newcommand{\fab}{\mathrm{(fab)}}
\newcommand{\abf}{\mathrm{(abf)}}
\newcommand{\abfp}{\mathrm{(abf')}}
\newcommand{\abcfp}{\mathrm{(abcf')}}
\newcommand{\abcdef}{\mathrm{(abcdef)}}

\newcommand{\Ks}{\mathsf{K}}
\newcommand{\Ss}{\mathsf{S}}
\newcommand{\Is}{\mathsf{I}}
\newcommand{\Bs}{\mathsf{B}}
\newcommand{\Cs}{\mathsf{C}}

\newcommand{\abs}[3]{\lbrack #2 \rbrack_{#1} . #3}
\newcommand{\trans}[1]{\mathrm{H}_{#1}}

\newcommand{\Opt}[1]{\mathrm{Opt}\llbracket#1\rrbracket}

\title{On the equivalence of different presentations of Turner's
  bracket abstraction algorithm}

\author{{\L}ukasz Czajka}

\date{13 Oct 2015}

\allowdisplaybreaks

\setlist{itemsep=0pt,topsep=\parsep}

\begin{document}
\maketitle

\begin{abstract}
  Turner's bracket abstraction algorithm is perhaps the most
  well-known improvement on simple bracket abstraction algorithms. It
  is also one of the most studied bracket abstraction algorithms. The
  definition of the algorithm in Turner's original paper is slightly
  ambiguous and it has been subject to different interpretations. It
  has been erroneously claimed in some papers that certain
  formulations of Turner's algorithm are equivalent. In this note we
  clarify the relationship between various presentations of Turner's
  algorithm and we show that some of them are in fact equivalent for
  translating lambda-terms in beta-normal form.
\end{abstract}

\section{Introduction}\label{sec_intro}

Bracket abstraction is a way of converting lambda-terms into a
first-order combinatory representation. This has applications to the
implementation of functional programming
languages~\cite{Turner1979,Jones1987,JoyRaywardBurton1985} or to the
automation in proof assistants~\cite{MengPaulson2008,Hurd2002}. The
well-known simple bracket abstraction algorithms of Curry and
Sch{\"o}nfinkel have the disadvantage of producing big combinatory
terms in the worst case. Perhaps the most well-known improvement on
these algorithms is the bracket abstraction algorithm of
Turner~\cite{Turner1979a}. It is also one of the most studied bracket
abstraction algorithms, with analysis of its worst-case and
average-case performance available in the literature.

The original definition of Turner's algorithm in~\cite{Turner1979} is
slightly ambiguous and it has been subject to many interpretations. It
seems to be a relatively prevalent misconception that certain
presentations of Turner's algorithm based on ``optimisation rules''
and others based on recursive equations with side conditions are
equivalent. This is explicitly (and erroneously) claimed
e.g.~in~\cite{JoyRaywardBurton1985}, and seems to be an implicit
assumption in many other papers.

The non-equivalence between a presentation of Turner's algorithm based
on ``optimisation rules'' and a presentation based on ``side
conditions'' was noted in~\cite{Bunder1990}. In this paper we clarify
the relationship between various formulations of Turner's
algorithm. In particular, we show the following.
\begin{itemize}
\item A certain formulation of Turner's algorithm based on
  ``optimisation rules'' is equivalent \emph{for translating
    lambda-terms in $\beta$-normal form} to a formulation based on
  equations with side conditions, provided one chooses the equations
  and the optimisations carefully. The formulations are not equivalent
  for translating lambda-terms with $\beta$-redexes.
\item If one removes ``$\eta$-rules'' then certain presentations
  become equivalent in general.
\item Analogous results hold for certain presentations of the
  Sch{\"o}nfinkel's bracket abstraction algorithm.
\end{itemize}

\section{Bracket abstraction}\label{sec_abstr}

In this section we give a general definition of a bracket abstraction
algorithm and present some of the simplest such algorithms. We assume
basic familiarity with the lambda-calculus~\cite{Barendregt1984}. We
consider lambda-terms up to $\alpha$-equivalence and we use the
variable convention.

First, we fix some notation and terminology. By~$\Lambda$ we denote
the set of all lambda-terms, by~$\Lambda_0$ the set of all closed
lambda-terms, and by~$V$ the set of variables. Given $B \subseteq
\Lambda_0$ by~$\CL(B)$ we denote the set of all lambda-terms built
from variables and elements of~$B$ using only application. A
\emph{(univariate) bracket (or combinatory) abstraction algorithm}~$A$
for a basis~$B \subseteq \Lambda_0$ is an algorithm computing a
function\footnote{Whenever convenient we confuse algorithms with the
  functions they compute.}  $A : V \times \CL(B) \to \CL(B)$ such that
for any $x \in V$ and $t \in \CL(B)$ we have $\FV(A(x,t)) = \FV(t)
\setminus \{x\}$ and $A(x,t) =_{\beta\eta} \lambda x . t$. We usually
write $\abs{A}{x}{t}$ instead of~$A(x,t)$. For an algorithm~$A$ the
\emph{induced translation} $\trans{A} : \Lambda \to \CL(B)$ is defined
recursively by
\[
\begin{array}{rcll}
  \trans{A}(x) &=& x \\
  \trans{A}(s t) &=& (\trans{A}(s)) (\trans{A}(t)) \\
  \trans{A}(\lambda x . t) &=& \abs{A}{x}{\trans{A}(t)}
\end{array}
\]
It follows by straightforward induction that $\FV(\trans{A}(t)) =
\FV(t)$ and $\trans{A}(t) =_{\beta\eta} t$.

Abstraction algorithms are usually presented by a list of recursive
equations with side conditions. It is to be understood that the first
applicable equation in the list is to be used.

For instance, the algorithm~$\fab$ of
Curry~\cite[\textsection6A]{CurryFeys1958} for the basis
$\{\Ss,\Ks,\Is\}$ where
\[
\begin{array}{rcl}
  \Ss &=& \lambda x y z . x z (y z) \\
  \Ks &=& \lambda x y . x \\
  \Is &=& \lambda x . x
\end{array}
\]
may be defined by the equations
\[
\begin{array}{rcl}
  \abs{\fab}{x}{s t} &=& \Ss (\abs{\fab}{x}{s}) (\abs{\fab}{x}{t}) \\
  \abs{\fab}{x}{x} &=& \Is \\
  \abs{\fab}{x}{t} &=& \Ks t
\end{array}
\]
The last equation is thus used only when the previous two cannot be
applied. For example $\abs{\fab}{x}{y y x} = \Ss (\Ss (\Ks y) (\Ks y))
\Is$. Note that we have $\abs{\fab}{x}{\Ss} = \Ss$,
$\abs{\fab}{x}{\Ks} = \Ks$ and $\abs{\fab}{x}{\Is} = \Is$, but
$\abs{\fab}{x}{\lambda x . t}$ is undefined if $\lambda x . t \notin
\{\Ss,\Ks,\Is\}$, because then $\lambda x . t \notin
\CL(\{\Ss,\Ks,\Is\})$. One easily shows by induction on the structure
of~$t \in \CL(\{\Ss,\Ks,\Is\})$ that indeed $\FV(\abs{\fab}{x}{t}) =
\FV(t) \setminus \{x\}$ and $\abs{\fab}{x}{t} \reduces_{\beta} t$,
so~$\fab$ is an abstraction algorithm. For all algorithms which we
present, their correctness, i.e., that they are abstraction
algorithms, follows by straightforward induction, and thus we will
avoid mentioning this explicitly every time.

The algorithm~$\abfp$ is defined by the equations for~$\fab$ plus the
optimisation rule
\[
\Ss (\Ks s) (\Ks t) \to \Ks (s t).
\]
More precisely the algorithm~$\abfp$ is defined by
\[
\begin{array}{rcl}
  \abs{\abfp}{x}{s t} &=& \Opt{\Ss (\abs{\abfp}{x}{s}) (\abs{\abfp}{x}{t})} \\
  \abs{\abfp}{x}{x} &=& \Is \\
  \abs{\abfp}{x}{t} &=& \Ks t
\end{array}
\]
where the function~$\Opt{}$ is defined by
\[
\begin{array}{rcl}
  \Opt{\Ss (\Ks s) (\Ks t)} &=& \Ks (s t) \\
  \Opt{\Ss s t} &=& \Ss s t
\end{array}
\]
Of course, it is to be understood that an earlier equation takes
precedence when more than one equation applies. This conforms to the
interpretation of ``optimisation rules''
in~\cite[Chapter~16]{Jones1987}, but e.g.~Bunder~\cite{Bunder1990}
interprets them as rewrite rules. For example we have
$\abs{\abfp}{x}{y y x} = \Ss (\Ks (y y)) \Is$. It is a recurring
pattern that certain bracket abstraction algorithms are defined by the
equations for~$\abfp$, differing only in the definition of the
function~$\Opt{}$. In such a case we will not repeat the equations
of~$\abfp$, and only note that an algorithm is \emph{defined by the
  optimisations} given by a function~$\Opt{}$.

The algorithm~$\abcfp$ is defined by the following optimisations.
\[
\begin{array}{rcl}
  \Opt{\Ss (\Ks s) (\Ks t)} &=& \Ks (s t) \\
  \Opt{\Ss (\Ks s) \Is} &=& s \\
  \Opt{\Ss s t} &=& \Ss s t
\end{array}
\]
Note that e.g.~$\abs{\abcfp}{x}{\Ss (\Ks y) (\Ks y) x} = \Ss (\Ks y)
(\Ks y)$.

It follows by induction on~$t$ that $\abs{\abfp}{x}{t} = \Ks t$ if $x
\notin \FV(t)$. Note that this would not be true if optimisations
could be applied as rewrite rules: below the root and recursively to
results of optimisations (consider e.g.~abstracting~$x$ from $\Ss (\Ks
a) (\Ks a)$). This seems to disprove\footnote{Like with many
  presentations of abstraction algorithms based on optimisation
  rewrite rules, it is not completely clear what the precise algorithm
  actually is, but judging by some examples given in~\cite{Bunder1990}
  recursive optimisations below the root of optimisation results are
  allowed.} claim~$(5)$ in~\cite{Bunder1990}.

Using the above fact one easily shows that~$\abfp$ is equivalent
(i.e.~gives identical results) to the following algorithm~$\abf$.
\[
\begin{array}{rcll}
  \abs{\abf}{x}{x} &=& \Is & \\
  \abs{\abf}{x}{t} &=& \Ks t & \text{if } x \notin \FV(t) \\
  \abs{\abf}{x}{s t} &=& \Ss (\abs{\abf}{x}{s}) (\abs{\abf}{x}{t}) &
\end{array}
\]

The algorithm~$\abf$ is perhaps the most widely known and also one of
the simplest bracket abstraction algorithms, but it is not
particularly efficient. A natural measure of the efficiency of an
abstraction algorithm~$A$ is the \emph{translation size} -- the size
of~$\trans{A}(t)$ as a function of the size of~$t$. For~$\fab$ the
translation size may be exponential, while for~$\abf$ it
is~$O(n^3)$. See~\cite{JoyRaywardBurton1985,Joy1984} for an analysis
of the translation size for various bracket abstraction
algorithms. For a fixed finite basis~$\Omega(n\log n)$ is a lower
bound on the translation
size~\cite{JoyRaywardBurton1985,Joy1984}. This bound is attained
in~\cite{KennawaySleep1987} (see also~\cite{Burton1982}
and~\cite[Section~4]{JoyRaywardBurton1985}).

Sch{\"o}nfinkel's bracket abstraction algorithm~$S$ is defined for the
basis $\{\Ss,\Ks,\Is,\Bs,\Cs\}$ where:
\[
\begin{array}{rcl}
  \Bs &=& \lambda x y z . x (y z) \\
  \Cs &=& \lambda x y z . x z y
\end{array}
\]
The algorithm~$S$ is defined by the following equations.
\[
\begin{array}{lrcll}
  (1)&\abs{S}{x}{t} &=& \Ks t & \text{if } x \notin \FV(t) \\
  (2)&\abs{S}{x}{x} &=& \Is & \\
  (3)&\abs{S}{x}{s x} &=& s & \text{if } x \notin
  \FV(s) \\
  (4)&\abs{S}{x}{s t} &=& \Bs s (\abs{S}{x}{t}) & \text{if } x \notin
  \FV(s) \\
  (5)&\abs{S}{x}{s t} &=& \Cs (\abs{S}{x}{s}) t & \text{if } x \notin
  \FV(t) \\
  (6)&\abs{S}{x}{s t} &=& \Ss (\abs{S}{x}{s}) (\abs{S}{x}{t}) &
\end{array}
\]
This algorithm is called~$\abcdef$
in~\cite[\textsection6A]{CurryFeys1958} and it is actually the
Sch{\"o}nfinkel's algorithm implicit in~\cite{Schonfinkel1924}. Like
for~$\abf$, the translation size for~$S$ is also~$O(n^3)$ but with a
smaller constant~\cite{JoyRaywardBurton1985}. A variant~$S'$ of
Sch{\"o}nfinkel's algorithm is defined by the optimisations:
\[
\begin{array}{lrcl}
  (1)&\Opt{\Ss (\Ks s) (\Ks t)} &=& \Ks (s t) \\
  (2)&\Opt{\Ss (\Ks s) \Is} &=& s \\
  (3)&\Opt{\Ss (\Ks s) t} &=& \Bs s t \\
  (4)&\Opt{\Ss s (\Ks t)} &=& \Cs s t \\
  (5)&\Opt{\Ss s t} &=& \Ss s t
\end{array}
\]
The algorithm~$S'$ seems to have been introduced by Turner
in~\cite{Turner1979a} where he calls it ``an improved algorithm of
Curry'', but this algorithm is not equivalent to any of Curry's
algorithms~\cite{Bunder1990}. In fact, it is a common misconception
(claimed e.g.~in~\cite{JoyRaywardBurton1985}) that the translations
induced by the algorithms~$S$ and~$S'$ are equivalent. As a
counterexample consider~$\lambda y . (\lambda z . x) y y$. We have
\[
\trans{S}(\lambda y . (\lambda z . x) y y) =
\abs{S}{y}{\Ks x y y} = \Ss (\abs{S}{y}{\Ks x y}) \Is =
\Ss (\Ks x) \Is
\]
but
\[
\trans{S'}(\lambda y . (\lambda z . x) y y) = \abs{S'}{y}{\Ks x y y} = x
\]
because
\[
\Opt{\Ss (\abs{S'}{y}{\Ks x y}) \Is} = \Opt{\Ss (\Ks x) \Is} = x.
\]
The difference is that the algorithm~$S'$ may effectively contract some
$\beta$-redexes already present in the input term. That the algorithms
themselves are not equivalent has already been observed by Bunder
in~\cite{Bunder1990} with the following counterexample: $\Ks \Ss x
(\Ks \Ss x)$. We have $\abs{S}{x}{\Ks \Ss x (\Ks \Ss x)} = \Ss
(\Ks \Ss) (\Ks \Ss)$ but $\abs{S'}{x}{\Ks \Ss x (\Ks \Ss x)} = \Ks (\Ss
\Ss)$. The term of Bunder's counterexample is indeed also a
counterexample for the equivalence of the induced translations, which
is not completely immediate, but it is easy to show using the
following identities (which do not hold e.g.~for~$\fab$):
\[
\begin{array}{l}
  \trans{S}(\Ks) = \trans{S'}(\Ks) = \Ks \\
  \trans{S}(\Ss) = \trans{S'}(\Ss) = \Ss
\end{array}
\]
As another counterexample consider the term $\lambda y . z ((\lambda x
. x) y)$. We have $\trans{S'}(\lambda y . z ((\lambda x . x) y)) = z$
but $\trans{S}(\lambda y . z ((\lambda x . x) y)) = \Bs z \Is$. This
shows that it may be impossible to rewrite~$\trans{S}(t)$
to~$\trans{S'}(t)$ using the optimisations of the algorithm~$S'$ as
rewrite rules.

\section{Turner's algorithm}\label{sec_turner}

Turner's algorithm~\cite{Turner1979a} is perhaps the most widely known
improvement on Sch{\"o}nfinkel's algorithm. The basis for Turner's
algorithm is $\{\Ss,\Ks,\Is,\Bs,\Cs,\Ss',\Bs',\Cs'\}$ where
\[
\begin{array}{rcl}
  \Ss' &=& \lambda k x y z . k (x z) (y z) \\
  \Bs' &=& \lambda k x y z . k x (y z) \\
  \Cs' &=& \lambda k x y z . k (x z) y
\end{array}
\]
Turner's algorithm~$T$ is defined by the following equations.
\[
\begin{array}{lrcll}
  (1)&\abs{T}{x}{t} &=& \Ks t & \text{if } x \notin \FV(t) \\
  (2)&\abs{T}{x}{x} &=& \Is & \\
  (3)&\abs{T}{x}{s x} &=& s & \text{if } x \notin
  \FV(s) \\
  (4)&\abs{T}{x}{u x t} &=& \Cs u t & \text{if } x \notin \FV(ut) \\
  (5)&\abs{T}{x}{u x t} &=& \Ss u (\abs{T}{x}{t}) & \text{if } x \notin \FV(u) \\
  (6)&\abs{T}{x}{u s t} &=& \Bs' u s (\abs{T}{x}{t}) & \text{if } x \notin
  \FV(u s) \\
  (7)&\abs{T}{x}{u s t} &=& \Cs' u (\abs{T}{x}{s}) t & \text{if } x \notin
  \FV(u t) \\
  (8)&\abs{T}{x}{u s t} &=& \Ss' u (\abs{T}{x}{s}) (\abs{T}{x}{t}) &
  \text{if } x \notin \FV(u) \\
  (9)&\abs{T}{x}{s t} &=& \Bs s (\abs{T}{x}{t}) & \text{if } x \notin
  \FV(s) \\
  (10)&\abs{T}{x}{s t} &=& \Cs (\abs{T}{x}{s}) t & \text{if } x \notin
  \FV(t) \\
  (11)&\abs{T}{x}{s t} &=& \Ss (\abs{T}{x}{s}) (\abs{T}{x}{t}) &
\end{array}
\]
The translation size of~$T$ is
worst-case~$O(n^2)$~\cite{JoyRaywardBurton1985} and
average-case~$O(n^{3/2})$~\cite{Hikita1984}.

The idea with Turner's combinators~$\Ss',\Bs',\Cs'$ is that they allow
to leave the structure of the abstract of an application $s t$
unaltered in the form $\kappa s' t'$ where~$\kappa$ is a ``tag''
composed entirely of combinators. For instance, if $x_1,x_2,x_3 \in
\FV(s) \cap \FV(t)$ (and~$s$ has a form such that the equations 3-5 in
the definition of~$T$ are not used) then
\[
\abs{T}{x_1,x_2,x_3}{s t} = \Ss' (\Ss' \Ss) (\abs{T}{x_1,x_2,x_3}{s})
(\abs{T}{x_1,x_2,x_3}{s})
\]
while
\[
\abs{S}{x_1,x_2,x_3}{s t} = \Ss (\Bs \Ss (\Bs (\Bs \Ss)
(\abs{S}{x_1,x_2,x_3}{s}))) (\abs{S}{x_1,x_2,x_3}{s}).
\]

In~\cite{Turner1979a} Turner formulates his algorithm in terms of
``optimisation rules''. There is some ambiguity in the original
definition and it has been subject to different
interpretations. Whatever the interpretation, the original formulation
is not equivalent to the algorithm~$T$ as defined above. Perhaps the
most common interpretation is like in~\cite[Chapter~16]{Jones1987}. We
thus define the algorithm~$T'$ by the following optimisations (see
Section~\ref{sec_abstr}).
\[
\begin{array}{lrcl}
  (1)&\Opt{\Ss (\Ks s) (\Ks t)} &=& \Ks (s t) \\
  (2)&\Opt{\Ss (\Ks s) \Is} &=& s \\
  (3)&\Opt{\Ss (\Ks (u s)) t} &=& \Bs' u s t \\
  (4)&\Opt{\Ss (\Ks s) t} &=& \Bs s t \\
  (5)&\Opt{\Ss (\Bs u s) (\Ks t)} &=& \Cs' u s t \\
  (6)&\Opt{\Ss (\Bs' u_1 u_2 s) (\Ks t)} &=& \Cs' (u_1 u_2) s t \\
  (7)&\Opt{\Ss s (\Ks t)} &=& \Cs s t \\
  (8)&\Opt{\Ss (\Bs u s) t} &=& \Ss' u s t \\
  (9)&\Opt{\Ss (\Bs' u_1 u_2 s) t} &=& \Ss' (u_1 u_2) s t \\
  (10)&\Opt{\Ss s t} &=& \Ss s t
\end{array}
\]
It is easily seen that the translations induced by~$T$ and~$T'$ are
not equivalent by reusing the counterexamples for~$S$ and~$S'$ from
the previous section. We will later show that the translations induced
by~$T$ and~$T'$ are equivalent for terms in $\beta$-normal form.

That~$T$ and~$T'$ are not equivalent is because they may effectively
perform $\eta$-contractions. If we disallow this, then the algorithms
become equivalent. Let~$T_{-\eta}$ be~$T$ without the equations~$(3)$,
$(4)$ and~$(5)$, and let~$T_{-\eta}'$ be~$T'$ without the
optimisation~$(2)$. We shall show later that~$T_{-\eta}$
and~$T_{-\eta}'$ are equivalent.

Another popular formulation of Turner's algorithm is~$T''$ which is
like~$T$ except that~$u$ is required to be closed. This formulation is
used e.g.~in~\cite{Hikita1984,Kennaway1982}. The translations induced
by the algorithms~$T'$ and~$T''$ are not equivalent even for terms in
$\beta$-normal form. For instance, we have
\[
\trans{T''}(\lambda x y z . y (x z) x) = \abs{T''}{x, y}{\Cs (\Bs y x)
  x} = \abs{T''}{x}{\Cs' \Cs (\Cs \Bs x) x} = \Ss' (\Cs' \Cs) (\Cs
\Bs) \Is
\]
but
\[
\trans{T'}(\lambda x y z . y (x z) x) = \abs{T'}{x, y}{\Cs' y x x} =
\abs{T'}{x}{\Cs (\Cs \Cs' x) x} = \Ss (\Bs \Cs (\Cs \Cs')) \Is.
\]
In~\cite{JoyRaywardBurton1985} it is erroneously claimed that an
algorithm~\mbox{Abs/Dash/1}, which is essentially~$T''$ with the
equations~$(4)$ and~$(5)$ removed and the equation~$(3)$ moved after
the equation~$(8)$, is
equivalent\footnote{In~\cite{JoyRaywardBurton1985} it is not
  completely clear what the precise algorithm based on ``optimisation
  rules'' actually is, but the ambiguity does not affect the fact that
  the formulations are not equivalent.} to~$T'$. Like~$T''$, the
algorithm~\mbox{Abs/Dash/1} is not equivalent to~$T'$ even for
translating lambda-terms in $\beta$-normal form.

In~\cite[Chapter~16]{Jones1987} it is suggested that Turner's
algorithm may be improved by using instead of~$\Bs'$ the
combinator~$\Bs^*$ defined by
\[
\Bs^* = \lambda f x y z . f (x (y z))
\]
Modified Turner's algorithm~$T_*$ is defined like algorithm~$T$ except
that the equation~$(6)$ is removed, the following equation is added
after the third one
\[
\begin{array}{rcll}
  \abs{T_*}{x}{s t} &=& \Bs^* s t_1 t_2 & \text{if } x \notin
  \FV(s) \text{ and } \abs{T_*}{x}{t} = \Bs t_1 t_2
\end{array}
\]
and the equation~$(9)$ is moved after the added one. A variant~$T_*'$
of modified Turner's algorithm is defined by the optimisations below.
\[
\begin{array}{lrcl}
  (1)&\Opt{\Ss (\Ks s) (\Ks t)} &=& \Ks (s t) \\
  (2)&\Opt{\Ss (\Ks s) \Is} &=& s \\
  (3)&\Opt{\Ss (\Ks u) (\Bs s t)} &=& \Bs^* u s t \\
  (4)&\Opt{\Ss (\Ks s) t} &=& \Bs s t \\
  (5)&\Opt{\Ss (\Bs u s) (\Ks t)} &=& \Cs' u s t \\
  (6)&\Opt{\Ss (\Bs^* u s_1 s_2) (\Ks t)} &=& \Cs' u (\Bs s_1 s_2) t \\
  (7)&\Opt{\Ss s (\Ks t)} &=& \Cs s t \\
  (8)&\Opt{\Ss (\Bs u s) t} &=& \Ss' u s t \\
  (9)&\Opt{\Ss (\Bs^* u s_1 s_2) t} &=& \Ss' u (\Bs s_1 s_2) t \\
  (10)&\Opt{\Ss s t} &=& \Ss s t
\end{array}
\]
Of course, the algorithms~$T_*$ and~$T_*'$ are not equivalent, which
may be seen by again considering the counterexamples against the
equivalence of~$S$ and~$S'$. We will show that~$T_*$ and~$T_*'$ are
equivalent as far as translating lambda-terms in $\beta$-normal form
is concerned. Actually, another variant~$T_*''$ of Turner's modified
algorithm presented in~\cite[Chapter~16]{Jones1987}, which is~$T_*'$
with the equations~$(6)$ and~$(9)$ removed. This algorithm~$T_*''$ is
not equivalent to~$T_*$ even for translating closed $\beta$-normal
forms. For instance, we have
\[
\trans{T_*}(\lambda x y . x (x (x y)) x) = \abs{T_*}{x}{\Cs' x (\Bs x
  x) x} = \Ss (\Ss \Cs' (\Ss \Bs \Is)) \Is
\]
but
\[
\trans{T_*''}(\lambda x y . x (x (x y)) x) = \abs{T_*''}{x}{\Opt{\Ss
    (\Bs^* x x x) (\Ks x)}} = \abs{T_*''}{x}{\Cs (\Bs^* x x x) x} =
\Ss' \Cs (\Ss (\Ss \Bs^* \Is) \Is) \Is.
\]

\section{Equivalence of~$\trans{T}$ and~$\trans{T'}$ for
  $\beta$-normal forms}\label{sec_equiv}

In this section we show that the translations induced by the
algorithms~$T$ and~$T'$ are equivalent for terms in $\beta$-normal
form.

\begin{lemma}\label{lem_t_0}~
  \begin{enumerate}
  \item If $x \notin \FV(t)$ then $\abs{T'}{x}{t} = \Ks t$.
  \item If $x \notin \FV(t)$ then $\abs{T'}{x}{t x} = t$.
  \end{enumerate}
\end{lemma}

\begin{proof}~
  \begin{enumerate}
  \item Induction on the structure of~$t$. If~$t$ is not an
    application then $\abs{T'}{x}{t} = \Ks t$ by the definition
    of~$T'$. So assume $t = t_1 t_2$. Then $\abs{T'}{x}{t} = \Opt{\Ss
      (\abs{T'}{x}{t_1}) (\abs{T'}{x}{t_2})}$. By the inductive
    hypothesis $\abs{T'}{x}{t_1} = \Ks t_1$ and $\abs{T'}{x}{t_2} =
    \Ks t_2$. Thus $\abs{T'}{x}{t} = \Opt{\Ss (\Ks t_1) (\Ks t_2)} =
    \Ks (t_1 t_2) = \Ks t$.
  \item Using the previous point we have $\abs{T'}{x}{t x} = \Opt{\Ss
      (\abs{T'}{x}{t}) \Is} = \Opt{\Ss (\Ks t) \Is} = t$.
  \end{enumerate}
\end{proof}

\begin{definition}
  A term is \emph{$T$-normal} if it does not contain subterms of the
  form $\Ks t_1 t_2$, $\Is t$, $\Bs t_1 t_2 t_3$ or $\Bs' t_1 t_2 t_3
  t_4$.
\end{definition}

\begin{lemma}\label{lem_t_1}
  Let~$t$ be $T$-normal.
  \begin{enumerate}
  \item If $\abs{T}{x}{t} = \Ks s$ then $s =
    t$ and $x \notin \FV(t)$.
  \item If $\abs{T}{x}{t} = \Is$ then $t = x$.
  \item If $\abs{T}{x}{t} = \Bs t_1 t_2$ then $t = t_1 t_2'$, $x
    \notin \FV(t_1)$ and $t_2 = \abs{T}{x}{t_2'}$.
  \item If $\abs{T}{x}{t} = \Bs' t_1 t_2 t_3$ then $t = t_1 t_2 t_3'$,
    $x \notin \FV(t_1t_2)$ and $t_3 = \abs{T}{x}{t_3'}$.
  \end{enumerate}
\end{lemma}

\begin{proof}
  We show the first point. It follows directly from the definition
  of~$T$ that either $s = t$ and $x \notin \FV(t)$, or $t = \Ks s x$
  with $x \notin \FV(s)$. The second case is impossible because~$t$ is
  $T$-normal.

  The proofs for the remaining points are analogous.
\end{proof}

\begin{lemma}\label{lem_t_2}
  If~$t$ is $T$-normal then $\abs{T}{x}{t} = \abs{T'}{x}{t}$.
\end{lemma}

\begin{proof}
  Induction on~$t$. We distinguish the cases according to which
  equation in the definition of~$T$ is used. If $x \notin \FV(t)$ then
  $\abs{T}{x}{t} = \Ks t = \abs{T'}{x}{t}$ by Lemma~\ref{lem_t_0}. If
  $t = x$ then $\abs{T}{x}{x} = \Is = \abs{T'}{x}{x}$. If $t = s x$
  with $x \notin \FV(s)$ then $\abs{T}{x}{t} = s = \Opt{\Ss (\Ks s)
    \Is} = \Opt{\Ss (\abs{T'}{x}{s}) (\abs{T'}{x}{x})} =
  \abs{T'}{x}{t}$, where in the penultimate equation we use
  Lemma~\ref{lem_t_0}.

  Assume $t = u x s$ and $x \notin \FV(us)$. Then $\abs{T}{x}{t} = \Cs
  u s$. On the other hand $\abs{T'}{x}{t} = \Opt{\Ss (\abs{T'}{x}{u
      x}) (\abs{T'}{x}{s})}$. We have $\abs{T'}{x}{s} = \Ks s$ and
  $\abs{T'}{x}{u x} = u$ by Lemma~\ref{lem_t_0}. It suffices to show
  $\Opt{\Ss u (\Ks s)} = \Cs u s$, for which it suffices that~$u$ does
  not have the form $\Ks u'$ or $\Bs u_1 u_2$. But this is the case
  because~$u x$ is $T$-normal.

  Assume $t = u x s$, $x \notin \FV(u)$ and $x \in \FV(s)$. Then
  $\abs{T}{x}{t} = \Ss u (\abs{T}{x}{s})$. We have $\abs{T}{x}{s} =
  \abs{T'}{x}{s}$ by the inductive hypothesis, and $\abs{T'}{x}{u x} =
  u$ by Lemma~\ref{lem_t_0}. Since $\abs{T'}{x}{t} = \Opt{\Ss
    (\abs{T'}{x}{u x}) (\abs{T'}{x}{s})} = \Opt{\Ss u
    (\abs{T}{x}{s})}$, it suffices to show that~$u$ does not have the
  form $\Ks u'$ or $\Bs u_1 u_2$, and $\abs{T}{x}{s}$ does not have
  the form $\Ks s'$. This follows from the fact that~$u x$ is
  $T$-normal and from Lemma~\ref{lem_t_1}.

  Assume $t = t_1 t_2 t_3$, $x \notin \FV(t_1t_2)$ and $x \in
  \FV(t_3)$, $t_3 \ne x$. Then $\abs{T}{x}{t} = \Bs' t_1 t_2
  (\abs{T}{x}{t_3})$. We have $\abs{T'}{x}{t_3} = \abs{T}{x}{t_3}$ by
  the inductive hypothesis, and $\abs{T'}{x}{t_1t_2} = \Ks (t_1 t_2)$
  by Lemma~\ref{lem_t_0}. Thus $\abs{T'}{x}{t} = \Opt{\Ss (\Ks (t_1
    t_2)) (\abs{T}{x}{t_3})}$, so it suffices to show that
  $\abs{T}{x}{t_3}$ does not have the form~$\Ks t_3'$ or~$\Is$. But if
  this is not the case then $x \notin \FV(t_3)$ or $t_3 = x$ by
  Lemma~\ref{lem_t_1}, which contradicts our assumptions.

  Assume $t = t_1 t_2 t_3$, $x \notin \FV(t_1t_3)$, $x \in \FV(t_2)$,
  $t_2 \ne x$. Then $\abs{T}{x}{t} = \Cs' t_1 (\abs{T}{x}{t_2})
  t_3$. We have $\abs{T'}{x}{t_2} = \abs{T}{x}{t_2}$ by the inductive
  hypothesis, and $\abs{T'}{x}{t_1} = \Ks t_1$, $\abs{T'}{x}{t_3} =
  \Ks t_3$ by Lemma~\ref{lem_t_0}. Thus~$\abs{T'}{x}{t_2}$ does not
  have the form~$\Ks t_2'$ or~$\Is$. First assume~$t_1$ is not an
  application. Then $\abs{T'}{x}{t_1 t_2} = \Opt{\Ss (\Ks t_1)
    (\abs{T'}{x}{t_2})} = \Bs t_1 (\abs{T'}{x}{t_2})$. Thus
  $\abs{T'}{x}{t} = \Opt{\Ss (\Bs t_1 (\abs{T'}{x}{t_2})) (\Ks t_3)} =
  \Cs' t_1 (\abs{T}{x}{t_2}) t_3 = \abs{T}{x}{t}$. If $t_1 = u_1 u_2$
  then $\abs{T'}{x}{t_1 t_2} = \Opt{\Ss (\Ks (u_1 u_2))
    (\abs{T'}{x}{t_2})} = \Bs' u_1 u_2 (\abs{T'}{x}{t_2})$. Thus
  $\abs{T'}{x}{t} = \Opt{\Ss (\Bs' u_1 u_2 (\abs{T'}{x}{t_2})) (\Ks
    t_3)} = \Cs' (u_1 u_2) (\abs{T}{x}{t_2}) t_3 = \abs{T}{x}{t}$.

  Assume $t = t_1 t_2 t_3$, $x \notin \FV(t_1)$, $x \in \FV(t_2)$, $x
  \in \FV(t_3)$, $t_2 \ne x$. Then $\abs{T}{x}{t} = \Ss' t_1
  (\abs{T}{x}{t_2}) (\abs{T}{x}{t_3})$. We have $\abs{T'}{x}{t_2} =
  \abs{T}{x}{t_2}$, $\abs{T'}{x}{t_3} = \abs{T}{x}{t_3}$ by the
  inductive hypothesis, and $\abs{T'}{x}{t_1} = \Ks t_1$ by
  Lemma~\ref{lem_t_0}. First assume~$t_1$ is not an application. Then
  like in the previous paragraph we obtain $\abs{T'}{x}{t_1 t_2} = \Bs
  t_1 (\abs{T'}{x}{t_2})$. Thus $\abs{T'}{x}{t} = \Opt{\Ss (\Bs t_1
    (\abs{T}{x}{t_2})) (\abs{T}{x}{t_3})}$. If $t_1 = u_1 u_2$ then
  like in the previous paragraph we obtain $\abs{T'}{x}{t} = \Opt{\Ss
    (\Bs' u_1 u_2 (\abs{T}{x}{t_2})) (\abs{T}{x}{t_3})}$. In each case
  it suffices to show that~$\abs{T}{x}{t_3}$ does not have the
  form~$\Ks t_3'$. This follows from Lemma~\ref{lem_t_1} and $x \in
  \FV(t_3)$.

  Assume $t = t_1 t_2$ and $\abs{T}{x}{t} = \Bs t_1
  (\abs{T}{x}{t_2})$. Then $x \notin \FV(t_1)$, $t_1$ is not an
  application (otherwise equation~$(6)$ would apply), $x \in \FV(t_2)$
  and $t_2 \ne x$. We have $\abs{T'}{x}{t_2} = \abs{T}{x}{t_2}$ by the
  inductive hypothesis, and $\abs{T'}{x}{t_1} = \Ks t_1$ by
  Lemma~\ref{lem_t_0}. Hence $\abs{T'}{x}{t} = \Opt{\Ss (\Ks t_1)
    (\abs{T}{x}{t_2})}$. Since~$t_1$ is not an application, it thus
  suffices to show that~$\abs{T}{x}{t_2}$ does not have the form~$\Ks
  t_2'$ or~$\Is$. But this follows from Lemma~\ref{lem_t_1}, $x \in
  \FV(t_2)$ and $t_2 \ne x$.

  Assume $t = t_1 t_2$ and $\abs{T}{x}{t} = \Cs (\abs{T}{x}{t_1})
  t_2$. Then $x \in \FV(t_1)$, $x \notin \FV(t_2)$, and if $t_1 =
  u_1u_2$ then $x \in \FV(u_1)$. We have $\abs{T'}{x}{t_1} =
  \abs{T}{x}{t_1}$ by the inductive hypothesis, and $\abs{T'}{x}{t_2}
  = \Ks t_2$ by Lemma~\ref{lem_t_0}. Thus $\abs{T'}{x}{t} = \Opt{\Ss
    (\abs{T}{x}{t_1}) (\Ks t_2)}$. So it suffices to show that
  $\abs{T}{x}{t_1}$ does not have the form~$\Ks t_1'$, $\Bs u_1 u_2$
  or $\Bs' u_1 u_2 u_3$. But this follows from Lemma~\ref{lem_t_1} and
  our assumptions on~$t_1$.

  Finally, assume $t = t_1 t_2$ and $\abs{T}{x}{t} = \Ss
  (\abs{T}{x}{t_1}) (\abs{T}{x}{t_2})$. Then $x \in \FV(t_1)$, $x \in
  \FV(t_2)$, and if $t_1=u_1u_2$ then $x \in \FV(u_1)$. We have
  $\abs{T'}{x}{t_1} = \abs{T}{x}{t_1}$ and $\abs{T'}{x}{t_2} =
  \abs{T}{x}{t_2}$ by the inductive hypothesis. Thus $\abs{T'}{x}{t} =
  \Opt{\Ss (\abs{T'}{x}{t_1}) (\abs{T'}{x}{t_2})}$. So it suffices to
  show that~$\abs{T}{x}{t_1}$ does not have the form~$\Ks t_1'$, $\Bs
  u_1 u_2$ or~$\Bs' u_1 u_2 u_3$, and~$\abs{T'}{x}{t_2}$ does not have
  the form~$\Ks t_2'$. This follows from Lemma~\ref{lem_t_1} and our
  assumptions on~$t_1$ and~$t_2$.
\end{proof}
 
\begin{lemma}\label{lem_t_3}
  If~$t$ is $T$-normal then so is $\abs{T}{x}{t}$.
\end{lemma}

\begin{proof}
  By induction on the structure of~$t$.
\end{proof}

\begin{lemma}\label{lem_t_4}
  If~$t$ is in $\beta$-normal form then~$\trans{T}(t)$ is $T$-normal.
\end{lemma}

\begin{proof}
  Induction on~$t$. Because~$t$ is in $\beta$-normal form, either $t =
  x t_1 \ldots t_n$ or $t = \lambda x . s$. In the first case the
  claim follows directly from the inductive hypothesis. In the second
  case we have $\trans{T}(t) = \abs{T}{x}{\trans{T}(s)}$. By the
  inductive hypothesis~$\trans{T}(s)$ is
  $T$-normal. Hence~$\trans{T}(t)$ is $T$-normal by
  Lemma~\ref{lem_t_3}.
\end{proof}

\begin{theorem}\label{thm_t}
  If $t$ is in $\beta$-normal form then $\trans{T}(t) =
  \trans{T'}(t)$.
\end{theorem}

\begin{proof}
  Induction on~$t$. If $t = x$ then $\trans{T}(t) = x =
  \trans{T'}(t)$. If $t = t_1 t_2$ then the claim follows directly
  from the inductive hypothesis. So assume $t = \lambda x . s$. Then
  $\trans{T}(t) = \abs{T}{x}{\trans{T}(s)}$. By the inductive
  hypothesis $\trans{T'}(s) = \trans{T}(s)$. By Lemma~\ref{lem_t_4} we
  have that~$\trans{T'}(s)$ is $T$-normal. Thus $\trans{T}(t) =
  \abs{T}{x}{\trans{T'}(s)} = \abs{T'}{x}{\trans{T'}(s)} =
  \trans{T'}(t)$ by Lemma~\ref{lem_t_2}.
\end{proof}

\section{Equivalence of~$T_{-\eta}$
  and~$T_{-\eta}'$}\label{sec_equiv_eta}

The reason for the non-equivalence of~$T$ and~$T'$ is the fact that
these algorithms may effectively perform some $\eta$-contractions. We
show that if this is disallowed, then the algorithms become
equivalent. In other words, we show that~$T_{-\eta}$ and~$T_{-\eta}'$
are equivalent (see Section~\ref{sec_turner}).

\begin{theorem}\label{thm_t_eta}
  For every lambda-term~$t$ and every variable~$x$ we have
  $\abs{T_{-\eta}}{x}{t} = \abs{T_{-\eta}'}{x}{t}$.
\end{theorem}

\begin{proof}
  Induction on~$t$. We consider possible cases according to which
  equation in the definition of~$\abs{T_{-\eta}}{x}{t}$ is used. If $x
  \notin \FV(t)$ or $t = x$ then $\abs{T_{-\eta}}{x}{t} =
  \abs{T_{-\eta}'}{x}{t}$ follows directly from definitions.

  The remaining cases are shown by a straightforward modification of
  the proof of Lemma~\ref{lem_t_2}, noting that for~$T_{-\eta}$ the
  first point of Lemma~\ref{lem_t_0} still holds, Lemma~\ref{lem_t_1}
  holds without the assumption that~$t$ is $T$-normal, and the only
  places in the proof of Lemma~\ref{lem_t_2} where the second point of
  Lemma~\ref{lem_t_0} (which does not hold), $T$-normality or
  optimisation~$(2)$ are used directly is for the equations~$(3)$,
  $(4)$ and~$(5)$ not present in~$T_{-\eta}$.

  By way of an example we consider the case when $t = t_1 t_2 t_3$ and
  $\abs{T_{-\eta}}{x}{t} = \Bs' t_1 t_2
  (\abs{T_{-\eta}}{x}{t_3})$. Then $x \notin \FV(t_1t_2)$ and $x \in
  \FV(t_3)$. By the inductive hypothesis $\abs{T_{-\eta}'}{x}{t_3} =
  \abs{T_{-\eta}}{x}{t_3}$ and $\abs{T_{-\eta}'}{x}{t_1t_2} =
  \abs{T_{-\eta}}{x}{t_1t_2} = \Ks (t_1 t_2)$ because $x \notin
  \FV(t_1t_2)$. Thus $\abs{T_{-\eta}'}{x}{t} = \Opt{\Ss (\Ks (t_1
    t_2)) (\abs{T_{-\eta}}{x}{t_3})}$, so it suffices to show that
  $\abs{T_{-\eta}}{x}{t_3}$ does not have the form~$\Ks t_3'$. But if
  $\abs{T_{-\eta}}{x}{t_3} = \Ks t_3'$ then $x \notin \FV(t_3)$, which
  contradicts our assumption.
\end{proof}

\begin{corollary}
  For every lambda-term~$t$ we have $\trans{T_{-\eta}}(t) =
  \trans{T_{-\eta}'}(t)$.
\end{corollary}

\section{Other equivalences}\label{sec_other_equiv}

By modifying the proof of Theorem~\ref{thm_t} from
Section~\ref{sec_equiv} we can show the following.

\begin{theorem}\label{thm_s}
  If~$t$ is in $\beta$-normal form then $\trans{S}(t) =
  \trans{S'}(t)$.
\end{theorem}

\begin{proof}[Proof sketch]
  The proof is a simplification of the proof of
  Theorem~\ref{thm_t}. One defines a term to be \emph{$S$-normal} if
  it does not contain subterms of the form $\Ks t_1 t_2$ or $\Is
  t$. Then one proves lemmas analogous to the lemmas proved in
  Section~\ref{sec_equiv}, essentially by simplifying the proofs in
  Section~\ref{sec_equiv}. The theorem then follows from the lemmas in
  exactly the same way.
\end{proof}

\begin{theorem}\label{thm_t_star}
  If $t$ is in $\beta$-normal form then $\trans{T_*}(t) =
  \trans{T_*'}(t)$.
\end{theorem}

\begin{proof}[Proof sketch]
  The proof is analogous to the proof of Theorem~\ref{thm_t} in
  Section~\ref{sec_equiv}. We define a term to be $T_*$-normal if it
  does not contain subterms of the form $\Ks t_1 t_2$, $\Is t$, $\Bs
  t_1 t_2 t_3$ or $\Bs^* t_1 t_2 t_3 t_4$, and we modify all lemmas
  appropriately. Only Lemma~\ref{lem_t_2} requires significant
  adjustments. First, we show the new case when $t = u s$ and
  $\abs{T_*}{x}{t} = \Bs^* u s_1 s_2$. Then $x \notin \FV(u)$ and
  $\abs{T_*}{x}{s} = \Bs s_1 s_2$. We have $\abs{T_*'}{x}{t} =
  \Opt{\Ss (\abs{T_*'}{x}{u}) (\abs{T_*'}{x}{s})}$. By the inductive
  hypothesis $\abs{T_*'}{x}{s} = \abs{T_*}{x}{s}$. Also
  $\abs{T_*'}{x}{u} = \Ks u$ by (an appropriate restatement of)
  Lemma~\ref{lem_t_0}. Thus $\abs{T_*'}{x}{t} = \Opt{\Ss (\Ks u) (\Bs
    s_1 s_2)} = \Bs^* u s_1 s_2 = \abs{T_*}{x}{t}$.

  We also need to reconsider the case when $\abs{T_*}{x}{t} = \Bs u
  s'$. Then $t = u s$, $x \notin \FV(u)$, $x \in \FV(s)$, $s \ne x$
  and~$s' = \abs{T_*}{x}{s}$ does not have the form $\Bs s_1 s_2$. We
  have $\abs{T_*'}{x}{t} = \Opt{\Ss (\abs{T_*'}{x}{u})
    (\abs{T_*'}{x}{s})}$. By Lemma~\ref{lem_t_0} and $x \notin \FV(u)$
  we have $\abs{T_*'}{x}{u} = \Ks u$. By the inductive hypothesis
  $\abs{T_*'}{x}{s} = \abs{T_*}{x}{s} = s'$. Hence $\abs{T_*'}{x}{t} =
  \Opt{\Ss (\Ks u) s'}$. Since~$s'$ does not have the form $\Bs s_1
  s_2$, the third optimisation of~$T_*'$ (the one for~$\Bs^*$) does
  not apply to $\Ss (\Ks u) s'$. Because $x \in \FV(s)$ and $s \ne x$,
  by (an appropriate restatement of) Lemma~\ref{lem_t_1} the first two
  optimisations do not apply either. Thus $\abs{T_*'}{x}{t} = \Bs u s'
  = \abs{T_*}{x}{t}$.

  Consider the case when $\abs{T_*}{x}{t} = \Cs' t_1 t_2' t_3$. Then
  $t = t_1 t_2 t_3$, $t_2' = \abs{T_*}{x}{t_2}$, $x \notin
  \FV(t_1t_3)$, $x \in \FV(t_2)$, $t_2 \ne x$. We have
  $\abs{T_*'}{x}{t_2} = \abs{T_*}{x}{t_2} = t_2'$ by the inductive
  hypothesis, and $\abs{T_*'}{x}{t_1} = \Ks t_1$, $\abs{T_*'}{x}{t_3}
  = \Ks t_3$ by Lemma~\ref{lem_t_0}. Thus~$t_2'$ does not have the
  form~$\Ks t_2'$ or~$\Is$. First assume~$t_2'$ does not have the form
  $\Bs u_1 u_2$. Then $\abs{T_*'}{x}{t_1 t_2} = \Opt{\Ss (\Ks t_1)
    t_2'} = \Bs t_1 t_2'$ and thus $\abs{T_*'}{x}{t} = \Opt{\Ss (\Bs
    t_1 t_2') (\Ks t_3)} = \Cs' t_1 t_2' t_3 = \abs{T_*}{x}{t}$. If
  $t_2' = \Bs u_1 u_2$ then $\abs{T_*'}{x}{t_1 t_2} = \Opt{\Ss (\Ks
    t_1) (\Bs u_1 u_2)} = \Bs^* t_1 u_1 u_2$. Thus $\abs{T_*'}{x}{t} =
  \Opt{\Ss (\Bs^* t_1 u_1 u_2) (\Ks t_3)} = \Cs' t_1 (\Bs u_1 u_2) t_3
  = \abs{T_*}{x}{t}$.

  The case when $\abs{T_*}{x}{t} = \Ss' t_1 t_2' t_3'$ needs
  adjustments similar to the case above. The remaning cases go through
  as before.
\end{proof}

Let~$S_{-\eta}$ be~$S$ without the equation~$(3)$, and
let~$S_{-\eta}'$ be~$S'$ without the optimisation~$(2)$. By modifying
the proof of Theorem~\ref{thm_t_eta} from Section~\ref{sec_equiv_eta}
we can show the following.

\begin{theorem}\label{thm_s_eta}
  For every lambda-term~$t$ and every variable~$x$ we have
  $\abs{S_{-\eta}}{x}{t} = \abs{S_{-\eta}'}{x}{t}$.
\end{theorem}

\begin{proof}[Proof sketch]
  The proof is a straightforward simplification of the proof of
  Theorem~\ref{thm_t_eta}.
\end{proof}

\begin{corollary}
  For every lambda-term~$t$ we have $\trans{S_{-\eta}}(t) =
  \trans{S_{-\eta}'}(t)$.
\end{corollary}

\bibliography{biblio}{}
\bibliographystyle{plain}

\end{document}